\documentclass[12pt]{amsart}
\usepackage[mathscr]{eucal}
\usepackage{latexsym, amssymb, amsthm, amsmath}
\usepackage{epsfig, xspace, pict2e}
\usepackage[svgnames]{xcolor}
\usepackage{ifthen}
\usepackage{young}

\newcommand{\CC}{\mathbb{C}}
\newcommand{\FF}{\mathbb{F}}
\newcommand{\ZZ}{\mathbb{Z}}

\newcommand{\RR}{\mathbb{R}}
\newcommand{\CP}{\mathbb{CP}}

\newcommand{\RP}{\mathbb{RP}}
\newcommand{\PP}{\mathbb{P}}

\newcommand{\<}{\langle}
\renewcommand{\>}{\rangle}

\newcommand{\SYT}{\mathsf{SYT}}
\newcommand{\Shift}{\mathsf{SST}}
\newcommand{\bolda}{{\bf a}}
\newcommand{\boldb}{{\bf b}}
\newcommand{\boldc}{{\bf c}}

\newcommand{\nth}{\ensuremath{^\text{th}}\xspace}

\newcommand{\polm}{\CC_{m-1}[z]}

\newcommand{\pol}[1]{\CC_{#1}[z]}

\newcommand{\Rect}{{\mathchoice%
{\raisebox{.08ex}{\large $\sqsubset\!\!\!\!\sqsupset$}}
{\raisebox{.08ex}{\large $\sqsubset\!\!\!\!\sqsupset$}}
{\sqsubset\!\!\!\!\sqsupset}
{\sqsubset\!\!\!\!\sqsupset}
}}
\newcommand{\Staircase}
{\,{\setlength\unitlength{1.75ex}%
\begin{picture}(1,1)(0,0)
\put(0,1){\line(1,0){1}}
\put(1,0){\line(0,1){1}}
\put(0,1){\line(1,-1){1}}
\end{picture}%
}\,}

\newcommand{\Wr}{\mathrm{Wr}}
\newcommand{\Gr}{\mathrm{Gr}}
\newcommand{\OGpol}{\mathrm{OG}(n,\pol{2n})}
\newcommand{\OGn}{\mathrm{OG}(n,2n{+}1)}

\newcommand{\sssbar}{{\scriptscriptstyle|}}
\newcommand{\tkinterval}{{(\sssbar\tau\sssbar, \sssbar\kappa\sssbar]}}

\newcommand{\defn}[1]{\textbf{\emph{#1}}}

\newtheorem{lemma}{Lemma}
\newtheorem{theorem}[lemma]{Theorem}

\newtheorem{proposition}[lemma]{Proposition}

\newtheorem{remark}[lemma]{Remark}

\newenvironment{restatetheorem}[1]
   {\begingroup \newtheorem*{theoremx}{#1}\begin{theoremx}}
   {\end{theoremx}\endgroup}

\title{The Wronski map and shifted tableau theory}
\author{Kevin Purbhoo}
\address{Department of Combinatorics \& Optimization \\
         University of Waterloo \\
         Waterloo, ON, N2L 3G1\\
         CANADA}
\email{kpurbhoo@math.uwaterloo.ca}
\urladdr{http://www.math.uwaterloo.ca/\~{}kpurbhoo}
\thanks{Research partially supported by an NSERC discovery grant.}

\begin{document}

\begin{abstract}
The Mukhin-Tarasov-Varchenko Theorem, conjectured by B. and M. Shapiro,
has a number of interesting consequences.  Among them
is a well-behaved correspondence between certain points on a 
Grassmannian --- those sent by the Wronski map to polynomials with only 
real roots --- and (dual equivalence classes of) Young tableaux.  

In this paper, we restrict this correspondence to the orthogonal 
Grassmannian $\OGn \subset \Gr(n,2n{+}1)$.  We prove that a point lies 
on $\OGn$ if and only if the corresponding tableau has a certain
type of symmetry.  From this we recover much of the theory of shifted 
tableaux for 
Schubert calculus on $\OGn$, including a new, geometric proof of the
Littlewood-Richardson rule for $\OGn$.
\end{abstract}

\maketitle 

\thispagestyle{plain}

\medskip

\section{Introduction}

\medskip

For any non-negative integer $k$,
let $\pol{k}$ denote the $(k{+}1)$-dimensional
complex vector space of polynomials of degree at most $k$:
$$\pol{k} := \{f(z) \in \FF[z] \mid \deg f(z) \leq k\}\,.$$

Let $X = \Gr(n, \pol{2n})$, the Grassmannian variety 
of all $n$-dimensional subspaces
of the $(2n{+}1)$-dimensional vector space $\pol{2n}$. 
If $x \in X$ is the span of polynomials $f_1(z), \dots, f_n(z)$, 
the Wronskian
\[
   \Wr(x;z) :=
   \begin{vmatrix}
   f_1(z) & \cdots & f_n(z)\\
   f_1'(z) & \cdots  & f_n'(z) \\
   \vdots &  \vdots & \vdots \\
   f_1^{(n-1)}(z) & \cdots & f_n^{(n-1)}(z)
   \end{vmatrix}\,.
\]
is a non-zero polynomial of degree at most $n(n{+}1)$, and up
to a scalar multiple, it depends only on $x$. 
Hence, $x \mapsto \Wr(x;z)$ determines a well-defined, 
morphism schemes $\Wr : X \to \PP\big(\pol{n(n+1)}\big)$ 
called the \defn{Wronski map}.  This morphism is flat and 
finite~\cite{EH}.

Let $\SYT(\Rect)$ denote the set of standard Young tableaux whose
shape is an $n \times (n{+}1)$ rectangle.
The degree of Wronski map is equal to $|\SYT(\Rect)|$; hence one
might hope to find a surjective correspondence between 
$\SYT(\Rect)$
and the points of a fibre $\Wr^{-1}(h(z))$ of the Wronski map.
It turns out that this is possible to do when the roots of $h(z)$
are all real; in this case, we write 
\[
  h(z) = \prod_{a_i \neq \infty}(z+a_i)\,,
\]
$a_1, \dots, a_{n(n+1)} \in \RP^1$ --- a polynomial of degree $n(n{+}1)-k$,
is considered to have a root of multiplicity $k$ at $\infty$.  
Eremenko and Gabrielov
first established such a correspondence in an asymptotic setting
\cite{EG}, and
the remarkable theorem of Mukhin, Tarasov and Varchenko 
(see Theorem~\ref{thm:MTV} in Section~\ref{sec:correspondence})
ensures that it can extended unambiguously to polynomials with only
real roots.  We refer the reader to the survey article~\cite{Sot-F} 
for a discussion of the history, context and other applications of 
this result.

In this paper, we will use the notation
$X(\bolda) := \Wr^{-1}\big(\prod_{a_i \neq \infty}(z+a_i)\big)$,
to denote the fibre of the Wronski map associated to the
multiset $\bolda = \{a_1, \dots, a_{n(n+1)}\}$, and $x_T(\bolda)$
to denote the specific point in $X(\bolda)$ that corresponds to
the tableau $T\in \SYT(\Rect)$.
We will review a characterization and other key properties of the 
correspondence in Section~\ref{sec:correspondence}.
For now it is enough to remark that if $n > 1$, 
$\bolda \mapsto x_T(\bolda)$ is \emph{not} a continuous function.
As strange as it may seem, this is a feature, not a bug: 
in~\cite{Pur-Gr}, we showed
that the discontinuities 
essentially encode Sch\"utzenberger's jeu de taquin, and this 
fact provides a tight connection between the geometry of $X$ and the 
combinatorics of Young tableaux.  

Our goal in this paper is to establish similar results for the orthogonal 
Grassmannian.
Let $\< \cdot, \cdot \>$ be the non-degenerate 
symmetric bilinear form on $\pol{2n}$ given by
$$\Big\< \sum_{k=0}^{2n} a_k \tfrac{z^k}{k!}\,,\,
\sum_{\ell=0}^{2n} b_\ell \tfrac{z^\ell}{\ell!} \Big\>
 = \sum_{m =0}^{2n} (-1)^m a_m b_{2n-m}\,.$$
The \defn{orthogonal Grassmannian} $Y = \OGpol \subset X$ is 
the variety of all $n$-dimensional isotropic subspaces of $\pol{2n}$.

The restriction of the Wronski map to $Y$ has the interesting
property that $\Wr(y;z)$ is a perfect square for all $y \in Y$
\cite{Pur-OG}.
This raises the following question.  Suppose that $x \in X$ has 
the property that $\Wr(x;z)$ is a perfect square. Under what
conditions can we conclude that $x \in Y$?

To give a concrete answer, we will need to assume, moreover, that 
$\Wr(x;z) = \prod_{a_i \neq \infty}(z+a_i)$ has only real roots.
This allows us to write $x=x_T(\bolda)$ for some 
$T \in \SYT(\Rect)$.  Suppose that the tableau $T$ has entry $k$ 
in row $i_k$ and column $j_k$.  
We'll say that $T$ is \defn{symmetrical} if 
$i_{2k}= j_{2k-1}$ and $j_{2k} = i_{2k-1}+1$, for all
$k=1, \dots, \frac{n(n+1)}{2}$.  See Figure~\ref{fig:symmtableau}
for an example.

\begin{figure}[tb]
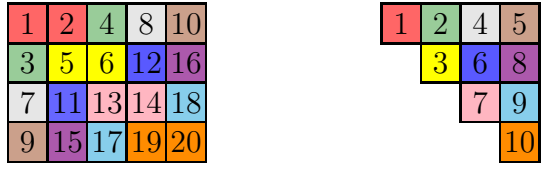

\begin{center}
  \begin{young}
  ![Red!60]1 & ![Red!60]2  & ![Green!40]4 & ![Grey!20]8 & ![Sienna!55]10\\
  ![Green!40]3 & ![Yellow]5 & ![Yellow]6 & ![Blue!65]12 & ![Purple!65]16 \\
  ![Grey!20]7 & ![Blue!60]11 & ![LightPink]13 & ![LightPink]14 & ![SkyBlue]18 \\
  ![Sienna!55]9 & ![Purple!65]15 & ![SkyBlue]17 & ![DarkOrange]19 & ![DarkOrange]20
  \end{young}
\qquad\qquad
  \begin{young}
  ,& ![Red!60]1 & ![Green!40]2 & ![Grey!20]4 & ![Sienna!55]5 \\
  ,& ,& ![Yellow]3 & ![Blue!65]6 & ![Purple!65]8 \\
  ,& ,& ,& ![LightPink]7 & ![SkyBlue]9 \\
  ,& ,& ,& ,& ![DarkOrange]10
  \end{young}
\end{center}
\caption{A symmetrical tableau (left), and the equivalent shifted 
tableau (right).}
\label{fig:symmtableau}
\end{figure}

We are now ready to state our main result, whose proof will be given
in Section~\ref{sec:proofmain}.
\begin{theorem}
\label{thm:Ysymm}
Let $x \in X$ be such that $\Wr(x;z) = \prod_{a_i \neq \infty}(z+a_i)$ 
is a perfect square with only real roots.  Then $x \in Y$ if and only 
if there exists a symmetrical tableau $T \in \SYT(\Rect)$ such that 
$x = x_T(\bolda)$.
\end{theorem}

The combinatorics of symmetrical tableaux are essentially the same
as the combinatorics of shifted tableaux; indeed if one deletes the
odd entries from a symmetrical tableau, the result is a 
standard shifted tableau (with entries multiplied by 2).  
In Section~\ref{sec:shifted} we will use
Theorem~\ref{thm:Ysymm} to show 
that the results of~\cite[Section 6]{Pur-Gr} have
analogues for $Y$, where tableaux are replaced by shifted tableaux.
This includes a geometric proof of the Littlewood-Richardson 
rule for $\OGn$.

We had already noted in~\cite{Pur-OG} that it should be possible to
prove these analogues by adapting the proofs in \cite{Pur-Gr}.  This, 
however, would be a long and tedious exercise.  The approach
we take in this paper is considerably more efficient.  Rather than
reprove everything, we will use Theorem~\ref{thm:Ysymm},
in combination with results from \cite{Pur-OG}, to deduce facts 
about $Y$ easily and directly from known facts about $X$.

\medskip

\section{Tableaux and points of $X$}
\label{sec:correspondence}

\medskip

Rather than recall exactly how the correspondence 
$(T, \bolda) \mapsto x_T(\bolda)$ was originally defined, we will state 
a theorem (Theorem~\ref{thm:correspondence}) that 
describes some of its important properties, and prove that these 
properties characterize the map.  Before we do this, we need some 
additional notation and background.

For each $a \in \CP^1$, define a full flag in $\pol{2n}$
\[
  F_\bullet(a) \ :\ 
  \{0\} \subset F_1(a) \subset \dots \subset F_{2n}(a) \subset \pol{2n}\,.
\]
For $a \in \CC$, $F_i(a) := (z+a)^{2n+1-i}\CC[z] \cap \polm$
is the set of
polynomials in $\pol{2n}$ divisible by $(z+a)^{2n+1-i}$.
We also set 
$F_i(\infty) := \pol{i-1} = \lim_{a \to \infty} F_i(a)$.

Let $\Lambda$ denote the set of all partitions
$\lambda : (\lambda^1 \geq \dots\geq \lambda^n)$, where $\lambda^1 \leq n{+}1$,
$\lambda^n \geq 0$.  The largest partition in $\Lambda$ is denoted by
$\Rect$.  For each $\lambda \in \Lambda$ we have
a \defn{Schubert variety} in $X$ relative to the flag $F_\bullet(a)$:
\[
  X_\lambda(a) 
  := \{x \in X \mid \dim \big(x \cap F_{n+1-\lambda^i+i}(a) \big) \geq i\,,
  \text{ for $i=1, \dots, n$}\}\,,
\]
which has codimension $|\lambda|$ in $X$.  We denote its cohomology
class by $[X_\lambda] \in H^{2|\lambda|}(X)$.

The relationship between Schubert varieties and the Wronski map
is given by the following classical fact (see e.g. \cite{EH,Pur-Gr,Sot-F}).

\begin{proposition}
\label{prop:Xschubertwronski}
$\Wr(x;z)$ is divisible by $(z+a)^k$ if and only
if $x \in X_\lambda(a)$ for some partition $\lambda \vdash k$.
\end{proposition}

The Mukhin-Tarasov-Varchenko Theorem asserts, moreover, that 
intersections of Schubert varieties relative to the flags $F_\bullet(a)$
are as well behaved as one might possibly hope.  

\begin{theorem}[Mukhin-Tarasov-Varchenko \cite{MTV1, MTV2}]
\label{thm:MTV}
If $a_1, \dots, a_s \in \RP^1$, and 
$\lambda_1, \dots \lambda_s \in \Lambda$ are partitions with
$|\lambda_1| + \dots + |\lambda_s| = \dim X$, then the
intersection
\[
   X_{\lambda_1}(a_1) \cap \dots \cap X_{\lambda_s}(a_s)
\]
is finite, transverse, every point in the intersection is real
(i.e. has a basis in $\RR[z]$).
\end{theorem}

We will also need some combinatorial notions from tableau theory.
If $\lambda, \mu$ are partitions, $\lambda \geq \mu$, let 
$\SYT(\lambda/\mu)$ denote the set of standard Young tableaux of
shape $\lambda/\mu$.  Suppose that $T \in \SYT(\lambda/\mu)$ and 
$U \in \SYT(\mu)$.  We can draw $T$ in red and $U$ in blue on the 
same diagram of shape $\lambda$, with $U$ ``inside'' of $T$.  
The basic jeu de taquin 
algorithm can be used to switch $U$ and $T$, so that we end up with
two new tableaux, $\hat{T}$ in red on the inside, and
$\hat{U}$ in blue on the outside.
\begin{enumerate}
\item Let $u$ be the largest entry in $U$.
\item Slide $u$ through $T$.
(If there are entries of $T$ to the right of $u$ and below
$u$, switch the smaller of these entries with $u$.  If only one
of these exists, switch it with $u$.  Repeat until $u$ has reached 
the ``outside'' of $T$.)
\item Let $u$ be the next largest entry in $U$, and repeat step (2)
until every entry of $U$ has been moved outside of $T$. 
\end{enumerate}
The resulting $\hat{T}$ is called the \defn{rectification} of $T$;
its shape is a partition, called the \defn{rectification shape} of $T$.
A theorem of Sch\"utzenberger states that $\hat{T}$ does not 
depend on on $U$~\cite{Sch}.  On the other hand, $\hat{U}$, may depend 
on $T$.  We say that $T$ and $T'$ are \defn{dual equivalent}, and write
$T \sim^* T'$, if $T$ and $T'$ produce the same $\hat{U}$ for
all (equivalently for some) $U \in \SYT(\mu)$.  Both versions of this
last definition are due to Haiman~\cite{Hai}.
It is worth noting that the dual equivalence relation $T \sim^* T'$ is 
quite different from
$\hat{T} = \hat{T'}$; in fact, if both are true then $T=T'$.
Dual equivalence classes on $\SYT(\lambda/\mu)$ are in bijection
with Littlewood-Richardson tableaux of shape $\lambda/\mu$; hence
statements involving Littlewood-Richardson numbers may be formulated in 
terms of counting dual equivalence classes.

If $T \in \SYT(\Rect)$, and $I$ is an 
interval, let $T_I$ denote the subtableau of $T$ consisting of entries 
in $I$.  We'll also sometimes write $T_{< i} := T_{[1,i)}$, 
$T_{\geq i} := T_{[i,n(n+1)]}$, etc.  $T_I$ is essentially a
standard Young tableau of some skew shape $\lambda/\mu$ --- the
definitions of rectification and dual equivalence make sense despite
the fact that the entries are
$\ZZ \cap I$ instead of $\{1, \dots, |\lambda/\mu|\}$.

Finally, let $A$ denote the set of $n(n{+}1)$-element multisets 
$\bolda = \{a_1, \dots, a_{n(n+1)}\}$, $a_1, \dots, a_{n(n+1)} \in \RP^1$.
There is a natural map $(\RP^1)^{n(n+1)} \to A$, $(a_1, \dots, a_{n(n+1)})
\mapsto \{a_1, \dots, a_{n(n+1)}\}$; we endow
$A$ with the quotient topology.  We will also need to refine the 
relation $|a| \leq |b|$, $a, b \in \RP^1$,
to a total order.  Any refinement will do, but for the sake of 
concreteness, define 
$a \preceq b$ if either $a = b$, $|a| < |b|$ or $0 < a = -b < \infty$.
Define a \defn{$\preceq$-zone} to be 
a subset of $A$ of the form
\[
 \big\{
 \{a_1 \preceq a_2 \preceq \dots \preceq a_{n(n+1)}\} \in A 
 \ \big|\  0 \leq a_i\epsilon_i \leq \infty
 \text{ for $i=1, \dots, n(n{+}1)$}
 \big\}\,,
\]
where $\epsilon_1, \dots, \epsilon_{n(n+1)} \in \{ \pm 1\}$.

\begin{theorem}
\label{thm:correspondence}
There is a unique map $\SYT(\Rect) \times A \to X$, denoted
$(T, \bolda) \mapsto x_T(\bolda)$,
with all of the following properties:
\begin{enumerate}
\item[(i)] For all $T \in \SYT(\Rect)$ and $\bolda \in A$, 
$x_T(\bolda) \in X(\bolda)$.
\item[(ii)] For all $\bolda \in A$, the map $T \mapsto x_T(\bolda)$ is 
surjective onto the fibre $X(\bolda)$.
If $\bolda$ is a set, i.e. $a_i \neq a_j$ for all $i \neq j$,
then it is also one to one.
\item[(iii)] For any $T \in \SYT(\Rect)$, the map $\bolda \mapsto
x_T(\bolda)$ is discontinuous at $\bolda$ only if 
$a_i = -a_j \notin \{0, \infty\}$ for some $i,j \in \{1, \dots, n(n{+}1)\}$.
More specifically, it is continuous on any $\preceq$-zone of $A$.
\item[(iv)]
Assume that $a_1 \preceq a_2 \preceq \dots \preceq a_{n(n+1)}$,
and that $a_i = a_{i+1} = \dots = a_j$.  Let $T \in \SYT(\Rect)$.
Then $x_T(\bolda) \in X_\lambda(a_i)$ where $\lambda$ is
the rectification shape of $T_{[i,j]}$.
\item[(v)]  Under the same hypotheses as (iv), let
$T, T' \in \SYT(\Rect)$ be two tableaux such that 
$T_{<i} = T'_{<i}$,
$T_{>j} = T'_{>j}$.
Then $x_T(\bolda) = x_{T'}(\bolda)$ if and only if
 $T_{[i,j]} \sim^* T'_{[i,j]}$.
\end{enumerate}
\end{theorem}

\begin{proof}
The ``existence'' part of this theorem is mainly a summary of several 
of the results in~\cite{Pur-Gr}.  There, a map satisfying (i) and (ii) 
is constructed for
\[
  \{\bolda \in A \mid
  0 < |a_1| < |a_2| < \dots < |a_{n(n+1)}| < \infty
  \}
\]
\cite[Corollary 4.10]{Pur-Gr}; it is continuous on that 
disconnected domain.  
Since $\Wr$ is flat and finite, we can extend this to a
continuous map on any single $\preceq$-zone, and
(ii) will still hold.  If $0,\infty \notin \bolda$, then
$\bolda$ is in a unique $\preceq$-zone, and this defines $x_T(\bolda)$
unambiguously.  Otherwise, $\bolda$ is in 
more than one $\preceq$-zone, and we need
\cite[Theorem 4.5]{Pur-Gr} to see that $x_T(\bolda)$ is
well-defined.  Thus the original correspondence can
be extended to all $\bolda \in A$ in such a way that (i)--(iii) hold.  

Statement (iv) and the $\Longleftarrow$ direction of (v) are the content of 
\cite[Theorem 6.4]{Pur-Gr}.   


As for the $\Longrightarrow$
direction of (v), suppose that
\[
  a_1 = \dots = a_{i_1} \prec a_{i_1+1} = \dots = a_{i_2} \prec
  \dots \prec a_{i_m+1} = \dots = a_{n(n+1)}\,.
\]
Let $\sim^*_\bolda$ be the equivalence relation
on $\SYT(\Rect)$ defined by $T \sim^*_\bolda T'$ if and only if
$T_{(i_l,i_{l+1}]} \sim^* T'_{(i_l,i_{l+1}]}$ for all $l = 0, 1, \dots, m$.
From (ii) and the $\Longleftarrow$ direction of (v), we know that 
$T \mapsto x_T(\bolda)$
is surjective onto $X(\bolda)$ and constant on the equivalence 
classes of $\sim^*_\bolda$.
The Littlewood-Richardson rule 
tells us that the number of $\sim^*_\bolda$ equivalence classes in 
$\SYT(\Rect)$ is equal to
\[
   \int_X \  \prod_{l=0}^m \bigg(
       \sum_{\lambda \vdash (i_{l+1}{-}i_l)} [X_\lambda]
   \bigg)\,.
\]
The transversality statement in Theorem~\ref{thm:MTV},
interpreted through Proposition~\ref{prop:Xschubertwronski},
asserts that this is exactly the number of points in $X(\bolda)$.  
Thus there
cannot be two equivalence classes of $\sim^*_\bolda$ that map to the
same point in $X(\bolda)$.

It remains to show uniqueness.
By the continuity property (iii), it is enough
to show that the inverse map $X(\bolda) \to \SYT(\Rect)$ is 
determined by properties (i)--(iv), in the case where $\bolda$ is
a set.  Assume that
\[
  a_1 \prec a_2 \prec \dots \prec a_{n(n+1)}\,,
\]
and let $x \in X(\bolda)$.  We will prove that from $x$, one can uniquely
determine the tableau $T$ such that $x = x_T(\bolda)$.

Let
$\bolda_{t,k} = \{ta_1, \dots, ta_k, a_{k+1}, \dots, a_{n(n+1)}\}$
for $t \in [0,1]$, $k \in \{1,\dots,n(n+1)\}$.  
By (ii) the map $T \mapsto x_T(\bolda_{t,k})$ is one to
one for all $t \in (0,1]$.  Thus there is a unique lifting of 
the path $t \mapsto \bolda_{t,k} \in A$, $t \in [0,1]$, to a path
 $t \mapsto x_{t,k} \in X(\bolda_{t,k})$,
with $x_{1,k} = x$.
By (iii), the
map $t \mapsto x_T(\bolda_{t,k})$ is also continuous on $[0,1]$,
and so we see that if $T$ is the tableau such that 
$x_T(\bolda) = x$, then $x_T(\bolda_{t,k}) = x_{t,k}$, for
all $t \in [0,1]$.  In particular, $x_T(\bolda_{0,k}) = x_{0,k}$.

Now, since $\bolda_{0,k}$ contains $0$ with multiplicity 
$k$,
by (iv) we have $x_T(\bolda_{0,k}) \in X_\lambda(0)$ where
$\lambda$ is the shape of $T_{\leq k}$.  Moreover, since $0$ does
not have multiplicity greater than $k$ in $\bolda_{0,k}$, we
cannot have $x_T(\bolda_{0,k}) \in X_\mu(0)$ for any 
$\mu > \lambda$.  It follows that the tableau $T$ such that
$x =x_T(\bolda)$ must have the property that
the shape of $T_{\leq k}$ is
the largest partition $\lambda$ such that 
$x_{0,k} \in X_\lambda(0)$.   Since this is true
for all $k \in \{1, \dots, n(n{+}1)\}$,
we have determined $T$, as required.
\end{proof}

\begin{remark} \rm
The fact that the discontinuities of the map $\bolda \mapsto x_T(\bolda)$
are at points where $a_i = -a_j$ for some $i,j$
has no particular geometric
significance: the fibres of the Wronski map are as well behaved at these
points as any.  However, the uniqueness of the map in 
Theorem~\ref{thm:correspondence} shows that it is 
impossible
to produce a continuous correspondence satisfying (i), (ii) and (iv).  
Since these are highly desirable properties, we are forced to have jump
discontinuities somewhere, and the points for which $a_i = -a_j$ are a 
fairly obvious and convenient choice for where to put them.
\end{remark}

\medskip

\section{Tableaux and points of $Y$}
\label{sec:proofmain}

\medskip

Our goal in this section is to prove Theorem~\ref{thm:Ysymm}.
We begin by recalling some of the relevant results 
from~\cite{Pur-OG}.

Let $\Sigma$ denote the set of all strict partitions
$\sigma :  (\sigma^1> \sigma^2 > \dots > \sigma^k)$,  with
$\sigma^1 \leq n$, $\sigma^k > 0$, $k \leq n$.
The diagram of $\sigma$ contains $\sigma_j$ boxes in the $j$\nth
row, with the leftmost box shifted $j-1$ boxes to the right.
A \defn{standard shifted tableau} of shape $\sigma$ is a filling of
the diagram of $\sigma$ with entries $1,\dots,|\sigma|$;
the entries must increase downwards and to the right; we write
$\Shift(\sigma)$ for the set of all standard shifted tableaux of 
shape $\sigma$.  
We will be particularly concerned with $\Shift(\Staircase)$, where
$\Staircase : (n > n-1 > \dots > 2 > 1)$ denotes
the largest strict partition in $\Sigma$.

The input data for a Schubert variety in $Y$ are a strict partition
$\sigma \in \Sigma$ and a flag $F_\bullet$ satisfying 
$\< F_i, F_{2n+1-i}\> = \{0\}$ for all $i=0, 1, \dots, 2n{+}1$; the 
flags $F_\bullet(a)$ 
satisfy this condition.
For our purposes, the most convenient way to define Schubert
varieties in $Y$ is in terms of the Schubert varieties in $X$.
For each strict partition
$\sigma \in \Sigma$, define a partition
$\widetilde \sigma \ :\ (\widetilde \sigma^1 \geq \widetilde \sigma^2 \geq 
\dots \geq 
\widetilde \sigma^n)$,
\[
  \widetilde \sigma^i := 
  \sigma^i + \#\{j  \mid j \leq i < j+\sigma^j\} \,.
\] 
The \defn{Schubert variety} in $Y$
relative to the flag $F_\bullet(a)$ is 
\[
  Y_\sigma(a) := Y \cap X_{\widetilde\sigma}(a)\,.
\]
$Y_\sigma(a)$ has 
codimension $|\sigma|$ in $Y$;
we denote its cohomology class by $[Y_\sigma] \in H^{2|\sigma|}(Y)$.

\begin{figure}[tb]
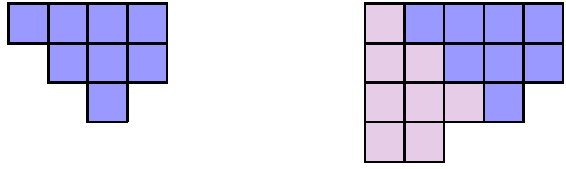

\begin{center}
\ysetshade{Blue!40}
\ysetaltshade{Purple!20}
\begin{young}[t]
! & ! & ! & ! \\
, & ! & ! & ! \\
, & , & !  \\
\end{young}
\qquad\qquad\qquad
\begin{young}[t]
? & ! & ! & ! & ! \\
? & ? & ! & ! & ! \\
? & ? & ? & !  \\
? & ? 
\end{young}
\end{center}
\caption{The strict partition $\sigma : (4 > 3 > 1)$ (left) and the associated 
partition $\widetilde\sigma : (5\geq 5\geq3 \geq 2\geq 2)$ (right).}
\label{fig:strict-symmetric}.
\end{figure}

Figure~\ref{fig:strict-symmetric} shows an example of a strict
partition $\sigma$ and the associated partition $\widetilde\sigma$.
The diagram of $\widetilde\sigma$ always decomposes as a copy
of the diagram of $\sigma$ and its ``transpose'', exhibiting the
same type of symmetry as a symmetrical standard young tableau.

If $\bolda = \{a_1, a_2, \dots, a_m\}$ is a multiset, let
$\bolda^\star := \{a_1, a_1, a_2, a_2, \dots, a_m, a_m\}$.
Let $A^\star \subset A$ denote the subspace of multisets
of the form $\bolda^\star$, $\bolda = \{a_1, \dots, a_{n(n+1)/2}\}$.
As mentioned in the
introduction, $\Wr(y;z)$ is always a perfect square for $y \in Y$,
hence $Y \cap X(\bolda) = \emptyset$ if $\bolda \notin A^\star$.
If $\bolda = \boldb^\star \in A^\star$, let
\[
   Y(\boldb) := Y \cap X(\boldb^\star)\,,
\]  
This will be the analogue of the fibre of the Wronski map for $Y$.

The next two results from~\cite{Pur-OG} 
are analogues of Proposition~\ref{prop:Xschubertwronski} and
Theorem~\ref{thm:MTV}.

\begin{proposition}
\label{prop:Yschubertwronski}
Let $y \in Y$.  $\Wr(y;z)$ is divisible by $(z+a)^{2k}$ if and only
if $y \in Y_\sigma(a)$ for some strict partition $\sigma \vdash k$.
\end{proposition}

\begin{theorem}
\label{thm:YMTV}
If $b_1, \dots b_s \in \RP^1$ are distinct real points, and
$\sigma_1, \dots \sigma_s \in \Sigma$, with
$|\sigma_1| + \dots + |\sigma_s| = \dim Y$, then the
intersection
$$Y_{\sigma_1}(b_1) \cap \dots \cap Y_{\sigma_s}(b_s)$$
is finite, transverse, and every point in the intersection is real.
\end{theorem}

In the case where $b_1, \dots, b_{n(n+1)/2}$ are distinct real
numbers and $\sigma_i = \yng[1ex](1)$ for $i=1, \dots, \frac{n(n+1)}{2}$,
Proposition~\ref{prop:Yschubertwronski} and Theorem~\ref{thm:YMTV}
tell us that $|Y(\boldb)|$ is equal to the Schubert intersection
number $\int_Y [Y_{\yng[0.7ex](1)}]^{n(n+1)/2}$.  Basic Schubert
calculus then gives us
\[
  |Y(\boldb)| = |\Shift(\Staircase)|\,.
\]

We need one additional Lemma.

\begin{lemma}
\label{lem:Ycomponents}
Let 
$\boldb = \{0 \prec b_1 \prec \dots \prec b_{n(n+1)/2}\} \subset \RR$
be a set, and
let $B$ be the space of multisets
$\boldc = \{c_1 \preceq \dots \preceq c_{n(n+1)/2}\}$ such that
$b_ic_i \geq 0$ for all $i$.
\begin{enumerate}
\item[(i)]
If $T \in \SYT(\Rect)$ is a tableau such that
$x_T(\boldb^\star) \in Y$, then for $x_T(\boldc^\star) \in Y$ for all
$\boldc \in B$.
\item[(ii)]
If $x \in Y(\boldc)$ where $\boldc \in B$, then there exists a tableau 
$T$ such that $x = x_T(\boldc^\star)$ and $x_T(\boldb^\star) \in Y$.
\end{enumerate}
\end{lemma}

\begin{proof}
Let $B^\circ \subset B$ be the subspace of sets
$\boldc = \{c_1 \prec \dots \prec c_{n(n+1)/2}\}$ such that
$b_ic_i > 0$ for all $i$.
The fibre $X(\boldc^\star)$ is reduced if $\boldc \in B^\circ$, 
and $Y(\boldc)$ is a subset of $X(\boldc^\star)$ varying continuously
with $\boldc$.  Since $B$ is connected, any continuous section
$s:B \to X$, $s(\boldc) \in X(\boldc^\star)$ will either 
have its image
entirely in $Y$, or $s(B^\circ) \cap Y = \emptyset$.  In particular,
since $b \in B^\circ$, if $s(b) \in Y$ then the former occurs.
Statement (i) follows by applying this to the section
$\boldc \mapsto x_T(\boldc^\star)$, which 
is continuous on $B$ by Theorem~\ref{thm:correspondence}(iii).

For statement (ii), choose any path $\boldc_t \in B$, $t\in [0,1]$,
with $\boldc_0 = \boldc$, $\boldc_1 = \boldb$.
We can lift this (though not necessarily uniquely) to a path in 
$y_t \in Y(\boldc_t)$ such that $y_0 = x$.  Then 
$y_1 = x_T(\boldb^\star)$ for some tableau $T$, and since
$\boldc \mapsto x_T(\boldc^\star)$ is continuous on $B$, we also have
$x = y_0 = x_T(\boldc^\star)$.
\end{proof}

We now turn to our main result.

\begin{restatetheorem}{Theorem \ref{thm:Ysymm}}
Let $x \in X(\bolda)$, where $\bolda \in A^\star$.
Then $x \in Y$ if and only 
if there exists a symmetrical tableau $T \in \SYT(\Rect)$ such that 
$x = x_T(\bolda)$.
\end{restatetheorem}

\begin{proof}
Let $T \in \SYT(\Rect)$, and let 
$\boldb = \{0 \prec b_1 \prec \dots \prec b_{n(n+1)/2}\} \subset \RR$
be a set.  We'll first show that if $x_T(\boldb^\star) \in Y$, then
$T$ is symmetrical.  

First, we note that since each element of $\boldb^\star$ has 
multiplicity $2$.  By Theorem~\ref{thm:correspondence}(iv),
the rectification shape of $T_{[2k-1,2k]}$ must be is equal
to the $\lambda$, where $x_T(\boldb^\star) \in X_\lambda(b_k)$.
Since $x_T(\boldb^\star) \in Y$, $\lambda = \yng[1ex](2)\,$.
For $T$ to have rectification shape $\yng[1ex](2)\,$, 
entry $2k-1$ in $T$ must be strictly to the left of entry $2k$.

Now, for $k = \{1, \dots, \frac{n(n+1)}{2}\}$, let
$\boldb_k = \{0, \dots, 0, b_{k+1}, \dots, b_{n(n+1)/2}\}$.
By Lemma~\ref{lem:Ycomponents}(i), we have $x_T(\boldb_k^\star) \in Y$.
Since $\boldb_k^\star$ contains $0$ with multiplicity $2k$, 
by Theorem~\ref{thm:correspondence}(iv), shape of $T_{\leq 2k}$ must
be equal to the largest partition $\lambda$ such that 
$x_T(\boldb_k^\star) \in X_\lambda(0)$.  Since 
$x_T(\boldb_k^\star) \in Y$, by $\lambda = \widetilde{\sigma}$ for
some strict partition $\sigma$.

Putting these two facts together, we see that $T$ must be symmetrical.
This shows that we have an injective map from $Y(\boldb)$ to
the symmetrical tableaux in $\SYT(\Rect)$,
or equivalently to $\Shift(\Staircase)$.
Since the number of points in the fibre $Y(\boldb)$ is 
equal to $|\Shift(\Staircase)|$, this is a bijection.

Finally, let $\bolda$ be as in the statement of the theorem.  Then
$\bolda = \boldc^\star$, for some multiset $\boldc$.  
We can find a set $\boldb$ 
as above such that $\boldc \in B$, where $B$ is the set defined in
Lemma~\ref{lem:Ycomponents}.  If $T$ is symmetrical, then since
$x_T(\boldb^\star) \in Y$, by Lemma~\ref{lem:Ycomponents}(i) we have 
$x_T(\boldc^\star) \in Y$.  Conversely, if $x \in X(\bolda) \cap Y = 
Y(\boldc)$, 
then by Lemma~\ref{lem:Ycomponents}(ii) there exists a tableau $T$ such 
that $x = x_T(\bolda^\star)$
and $x_T(\boldb^\star) \in Y$.  As we've just shown, the last statement 
implies that $T$ is 
symmetrical.
\end{proof}

\medskip

\section{Shifted tableau theory}
\label{sec:shifted}

\medskip

The theory of shifted tableaux, as developed
in~\cite{Hai, Pra, Sag, Ste, Wor}, is parallel to the theory of Young
tableaux.  In particular, the jeu de taquin theory works in essentially
the same way.
Given strict partitions $\sigma \geq \tau$ in $\Sigma$, 
and shifted tableaux
$T \in \Shift(\sigma/\tau)$ and $U \in \Shift(\tau)$, 
the tableau switching algorithm outlined in 
Section~\ref{sec:correspondence} makes sense exactly as stated.
Thus we can define the rectification (and rectification shape)
of $T \in \Shift(\sigma/\tau)$
as well as the dual equivalence relation $\sim^*$ on $\Shift(\sigma/\tau)$.

As already noted in the introduction, for any shifted tableau 
$T \in \Shift(\Staircase)$, there is is
a corresponding symmetrical tableau $T^\star \in \SYT(\Rect)$,
characterized by the fact that deleting the odd entries 
of $T^\star$ gives $T$, with entries multiplied by $2$.  
This same definition makes sense for
any $T \in \Shift(\sigma/\tau)$ of skew shape, in which case
the corresponding $T^\star$ is a skew tableau in 
$\SYT(\widetilde\sigma/\widetilde\tau)$.

\begin{lemma}
\label{lem:shiftedsymmetric}
Let $T \in \Shift(\sigma/\tau)$ and $U \in \Shift(\tau)$.
Let $\hat T$ and $\hat U$ be the results of applying the tableau switching
algorithm to $T$ and $U$.  Let $\widehat{T^\star}$
and $\widehat{U^\star}$ be the results of applying the 
switching algorithm to $T^\star \in \SYT(\widetilde\sigma/\widetilde\tau)$
and $U^\star \in \SYT(\widetilde\tau)$.  Then
\[
   \widehat{T^\star} = (\hat T)^\star \qquad \text{and} \qquad
   \widehat{U^\star} = (\hat U)^\star\,.
\]
\end{lemma}

\begin{proof}
One can easily check that each time we slide \emph{two} entries 
of $U^\star$ through
$T^\star$, the first never crosses below the diagonal,
and the second entry takes a path symmetrical to the first.
Thus if $T^\star$ and $U^\star$ are remain symmetrical throughout
the switching algorithm, and if we simply delete the odd entries,
we recover the switching algorithm for $T$ and $U$.
\end{proof}

From this observation, we can deduce facts about dual
equivalence for shifted tableaux from the corresponding facts about
standard Young tableaux.  For example:

\begin{proposition}
\label{prop:specialdualequiv}
Let $\tau \in \Sigma$.
Any two tableaux in $\Shift(\tau)$ are dual equivalent, as
are any two tableaux in $\Shift(\Staircase/\tau)$.
\end{proposition}

The analogous statement for $\mu \in \Lambda$ is that
any two tableaux in $\SYT(\mu)$ are dual
equivalent, as are any two tableaux in $\SYT(\Rect/\mu)$.  This,
and Proposition~\ref{prop:specialdualequiv} are proved 
combinatorially in~\cite{Hai}.
However, statements about dual equivalence for standard Young
tableaux have a geometric interpretation.  For example, the assertion
above follows from \cite[Lemma 6.1]{Pur-Gr} 
and \cite[Theorem 6.4]{Pur-Gr}; the equivalence of the two versions of
the definition
of $\sim^*$ given in Section~\ref{sec:correspondence} is part of the 
proof of \cite[Theorem 6.4]{Pur-Gr}.
This motivates us to outline a
short proof of Proposition~\ref{prop:specialdualequiv} by reduction.

\begin{proof}
The statement for $\Shift(\tau)$ is immediate from the definition
of dual equivalence.  For $\Shift(\Staircase/\tau)$,
Lemma~\ref{lem:shiftedsymmetric} implies that
$T \sim^* T' \in \Shift(\Staircase/\tau)$ if and only
if $T^\star \sim^* (T')^\star \in \SYT(\Rect/\widetilde{\tau})$
(for this, we need both definitions of the dual equivalence relation,
one for each direction).
Since any two tableaux in $\SYT(\Rect/\widetilde{\tau})$ are dual
equivalent, the result follows.
\end{proof}

For $T \in \Shift(\Staircase)$, and $\boldb = \{b_1, \dots, b_{n(n+1)}\}$,
$b_1, \dots, b_{n(n+1)} \in \RP^1$, let
$y_T(\boldb) := x_{T^\star}(\boldb^\star)$.
Theorem~\ref{thm:Ysymm} tells us 
that $y_T(\boldb) \in Y(\boldb)$, 
and every 
point in $Y(\boldb)$ is of the form $y_T(\boldb)$ for some standard 
shifted 
tableau $T$.  This key fact will be used implicitly throughout the
rest of the paper.

\begin{theorem}
\label{thm:Ydualequiv}
Suppose that $b_1 \preceq b_2 \preceq \dots \preceq b_{n(n+1)/2}$,
and that $b_i = b_{i+1} = \dots = b_j$.  
\begin{enumerate}
\item[(i)]
Let $T \in \Shift(\Staircase)$.
Then $y_T(\boldb) \in Y_\sigma(b_i)$ where $\sigma$ is
the rectification shape of $T_{[i,j]}$.
\item[(ii)]  Let $T, T' \in \Shift(\Staircase)$ be two tableaux such 
that 
$T_{<i} = T'_{<i}$,
$T_{>j} = T'_{>j}$.
Then $y_T(\boldb) = y_{T'}(\boldb)$ if and only if
$T_{[i,j]} \sim^* T'_{[i,j]}$.
\end{enumerate}
\end{theorem}

\begin{proof}
For (i), by Theorem~\ref{thm:correspondence}(iv), 
$y_T(\boldb) = x_{T^\star}(\boldb^\star) \in X_\lambda(b_i)$
where $\lambda$ is the rectification shape of $T^\star_{[2i-1,2j]}$.
By Lemma~\ref{lem:shiftedsymmetric}, $\lambda = \widetilde\sigma$
where $\sigma$ is the rectification shape of $T_{[i,j]}$.  
Since we also know $y_T(\boldb) \in Y$, we deduce
$y_T(\boldb) \in X_\lambda(b_i) \cap Y = Y_\sigma(b_i)$.

For (ii), we have
$y_T(\boldb) = x_{T^\star}(\boldb^\star)$ and
$y_{T'}(\boldb) = x_{(T')^\star}(\boldb^\star)$.
By Theorem~\ref{thm:correspondence}(v),
$x_{T^\star}(\boldb^\star) = x_{(T')^\star}(\boldb^\star)$ if and
only if
$T^\star_{[2i-1,2j]} \sim^* (T')^\star_{[2i-1,2j]}$.  
By Lemma~\ref{lem:shiftedsymmetric} this is true if and only if
$T_{[i,j]} \sim^* T'_{[i,j]}$.
\end{proof}

\begin{remark} \rm
A related result, \cite[Theorem 6.2]{Pur-Gr}, explains the geometric
significance of the equivalence relation $\sim$ on standard Young tableaux, 
defined by
$T \sim T'$ iff the rectifications of $T$ and $T'$ are equal.  
This too has an analogue
for $Y$, which can be proved by similar arguments.  Since new
notation is required to state it, we will omit further details here.

Similarly, \cite[Theorem 3.5]{Pur-Gr}, which describes how the
correspondence $(T, \bolda) \mapsto x_T(\bolda)$ changes at points 
of discontinuity, has an analogue for $Y$.  The statement is virtually
identical, but with tableaux replaced by shifted tableaux.  
Here, however, a bit more finesse is required in the proof, 
since a path in $A^\star$ does not satisfy the hypotheses of 
\cite[Theorem 3.5]{Pur-Gr}.  This can be resolved by perturbing
the path, and we leave the details to the reader.
\end{remark}

As an application of Theorem~\ref{thm:Ydualequiv}, we prove a version of 
the Littlewood-Richardson rule for the orthogonal Grassmannian
(Theorem~\ref{thm:YLR}).

\begin{lemma}
\label{lem:Yrectshape}
For $\kappa \in \Sigma$, let $\kappa^\vee \in \Sigma$ be the
strict partition such that $\int_Y [Y_\kappa] \cdot [Y_{\kappa^\vee}] = 1$
($[Y_{\kappa^\vee}]$ is dual to $[Y_\kappa]$ under the Poincar\'e pairing).
Every tableau in $\Shift(\Staircase/\kappa)$ has rectification shape 
$\kappa^\vee$.
\end{lemma}

\begin{proof}
Let $T \in \Shift(\Staircase)$ be a tableau such that $T_{\leq |\kappa|}$
has shape $\kappa$, and
suppose $T_{> |\kappa|}$ has rectification shape $\sigma$.
Let $\boldb = \{b_1, \dots, b_{n(n+1)}\}$ where
$ b_1 = \dots = b_{|\kappa|} = 0$ and 
$b_{|\kappa|+1} = \dots = b_{n(n+1)/2} = \infty$.
By Theorem~\ref{thm:Ydualequiv}(i), $y_T(\boldb) \in 
Y_\kappa(0) \cap Y_\sigma(\infty)$, but the fact that this intersection is
non-empty implies $\sigma = \kappa^\vee$.
\end{proof}

\begin{theorem}[Littlewood-Richardson rule for $\OGn$]
\label{thm:YLR}
For $\sigma, \tau, \kappa \in \Sigma,$
the Littlewood-Richardson number $c_{\sigma\tau}^\kappa$ for
$\OGn$, defined by 
\[
   [Y_\sigma]\cdot [Y_\tau] = \sum_\kappa c_{\sigma\tau}^\kappa [Y_\kappa]
\]
in $H^*(Y)$, is equal to the number dual equivalence classes 
in $\Shift(\kappa/\tau)$ with rectification shape $\sigma$.
\end{theorem}

\begin{proof}
With $\kappa^\vee$ as in Lemma~\ref{lem:Yrectshape}, we have
\[
  c_{\sigma\tau}^\kappa 
  = \int_Y [Y_\sigma] \cdot [Y_\tau] \cdot [Y_{\kappa^\vee}]\,.
\]
By Theorem~\ref{thm:YMTV}, this is the number of points in
\[
   Y_\tau(0) \cap Y_\sigma(1) \cap Y_{\kappa^\vee}(\infty)\,.
\]
Theorem~\ref{thm:Ydualequiv} allows to determine exactly which
tableaux correspond to points in this intersection, and when two
tableaux correspond to the same point.

Let $\boldb = \{b_1, \dots, b_{n(n+1)}\}$ where
\[
  b_1 = \dots = b_{|\tau|} = 0\,, \quad
  b_{|\tau|+1} = \dots = b_{|\kappa|} = 1\,,
  \quad\text{and}\quad
  b_{|\kappa|+1} = \dots = b_{n(n+1)/2} = \infty\,,
\]
and let $T \in \Shift(\Staircase)$.  
By Theorem~\ref{thm:Ydualequiv}(i) we have:
\begin{itemize}
\item $y_T(\boldb) \in Y_\tau(0)$ $\iff$
$T_{\leq |\tau|}$ has shape $\tau$; 
\item $y_T(\boldb) \in Y_\sigma(1)$ $\iff$ $T_\tkinterval$ has 
rectification shape $\sigma$;
\item $y_T(\boldb) \in Y_{\kappa^\vee}(\infty)$ $\iff$
$T_{>|\kappa|}$ has rectification shape $\kappa^\vee$, or equivalently
by Lemma~\ref{lem:Yrectshape},
$T_{\leq|\kappa|}$ has shape $\kappa$.
\end{itemize}
In other words 
$y_T(\boldb)  \in 
Y_\tau(0) \cap Y_\sigma(1) \cap Y_{\kappa^\vee}(\infty)$ if and only if
$T_\tkinterval$ has shape $\kappa/\tau$ and rectification
shape $\sigma$.

Moreover, by Theorem~\ref{thm:Ydualequiv}(ii), $T$ and $T'$ correspond to 
the same point if and only if
$T_{\leq |\tau|}  \sim^* T'_{\leq |\tau|}$,
$T_\tkinterval \sim^* T'_\tkinterval$ and
$T_{>|\kappa|} \sim^* T'_{>|\kappa|}$.
By Proposition~\ref{prop:specialdualequiv},
the first and last of these are true whenever the subtableaux have the
same shape.  Thus $T$ and $T'$ correspond to the same point if
and only if $T_\tkinterval \sim^* T'_\tkinterval$.

These two arguments show that the point $y_T(\boldb)$ depends only on
$T_\tkinterval$.  Putting them together, the points in
$Y_\tau(0) \cap Y_\sigma(1) \cap Y_{\kappa^\vee}(\infty)$
correspond bijectively to tableaux in $\Shift(\kappa/\tau)$ with
rectification shape $\sigma$, modulo dual equivalence.
\end{proof}



\end{document}